\documentclass[
]{amsart}
\usepackage{algorithm, algorithmic}
\usepackage{amsmath}\usepackage{amsthm}\usepackage{amssymb}\usepackage{mathrsfs}\usepackage{amscd}\usepackage{graphicx}\usepackage{subfigure}\usepackage{amsfonts}\usepackage{amsxtra}\usepackage{color}
\usepackage{amscd}

\DeclareMathOperator{\trace}{trace}
\DeclareMathOperator{\Prob}{Prob}

\DeclareMathOperator{\diag}{diag}\DeclareMathOperator*{\spann}{span}\DeclareMathOperator{\supp}{supp}

\DeclareMathOperator{\Pol}{Pol}

\newcommand{\EE}{\mathbb{E}}

\theoremstyle{definition}
\newtheorem{definition}{Definition}[section]
\newtheorem{example}[definition]{Example}\newtheorem{remark}[definition]{Remark}
\theoremstyle{plain}\newtheorem{theorem}[definition]{Theorem}\newtheorem{lemma}[definition]{Lemma}\newtheorem{corollary}[definition]{Corollary}\newtheorem{proposition}[definition]{Proposition}

\newcommand{\Z}{\mathbb{Z}}\newcommand{\R}{\mathbb{R}}\newcommand{\C}{\mathbb{C}}\newcommand{\N}{\mathbb{N}}

\begin{document}

\title[Phase retrieval using positive semi-definite matrices]{Phase retrieval using random cubatures and fusion frames of positive semidefinite matrices
}

\author[M.~Ehler]{Martin Ehler}
\address[M.~Ehler]{University of Vienna,
Department of Mathematics,
Oskar-Morgenstern-Platz 1
A-1090 Vienna
}
\email{martin.ehler@univie.ac.at}

\author[M.~Gr\"af]{Manuel Gr\"af}
\address[M.~Gr\"af]{University of Vienna,
Department of Mathematics,
Oskar-Morgenstern-Platz 1
A-1090 Vienna
}
\email{manuel.graef@univie.ac.at}

\author[F.~J.~Kir\'aly]{Franz J.~Kir\'aly}

\address[F.~J.~Kir\'aly]{Department of Statistical Science, University College London}
\email{f.kiraly@ucl.ac.uk}

\begin{abstract}

As a generalization of the standard phase retrieval problem, we seek to reconstruct symmetric rank-$1$ matrices from inner products with subclasses of positive semidefinite matrices. For such subclasses, we introduce random cubatures for spaces of multivariate polynomials based on moment conditions. The inner products with samples from sufficiently strong random cubatures  allow the reconstruction of symmetric rank-$1$ matrices with a decent probability by solving the feasibility problem of a semidefinite program.



\end{abstract}



\maketitle


\section{Introduction}

Many signal processing problems in engineering such as X-ray crystallography and coherent diffraction imaging require the reconstruction of symmetric rank-$1$ matricies from inner products with rank-$1$ projectors, often called phase retrieval, cf.~\cite{Balan:2006fk,Candes:2011fk} and references therein. Signal recovery from inner products with higher rank positive semidefinite matrices is a suitable model when diffraction patterns are weighted averages of $k$ wavefields, which occurs with incoherent diffraction \cite{Elser:2008fk}.

Classical reconstruction algorithms for the rank-$1$ phase retrieval problem are based on iterated projection schemes \cite{Fienup:1982vn,Gerchberg:1972kx} but there is a lack of stringent mathematical recovery guarantees.
 Recently, signal reconstruction with high probability is guaranteed in \cite{Candes:2011fk,Candes:uq} by solving the feasibility problem of a semidefinite program when sufficiently many rank-$1$ projectors are chosen in a uniformly distributed fashion, see also \cite{Waldspurger:2012uq}. Similarly, higher rank phase retrieval was solved in \cite{Bachoc:2012fk} by using uniformly distributed rank-$k$ orthogonal projectors.

To better match the measurement process in optical physics, the requirement of uniform distribution must be relaxed. For $k=1$, such an important relaxation was recently obtained in \cite{Gross:2013fk}, where random sampling rank-$1$ projectors from so-called spherical designs of strength $t\geq 3$ has been addressed. Increasing $t$ yields higher recovery probability and allows for fewer measurements. Asymptotic existence results of strong spherical designs were obtained in \cite{Bondarenko:2011kx,Bondarenko:2010aa}. Deriving actual constructions, however, is a delicate issue. Recently, \cite{Kueng:2014rw} overcomes of such issues with results for $k=1$ involving so-called cubatures (weighted designs) with strength $t\geq 4$, whose existence is well-understood and many linear algebra based constructions are known \cite{Harpe:2005fk}. 

Here, we shall generalize \cite{Gross:2013fk} to the range of positive semidefinite measurement matrices, and we do not require designs but only so-called cubatures. In contrast to \cite{Kueng:2014rw}, we only need strength $t\geq 3$.   
We address the real setting and point out some specialities that are due to the higher rank, partially based on earlier observations in \cite{Bachoc:2012fk,Bachoc:2010aa}.
%
To summarize, we generalize the results in \cite{Gross:2013fk} from rank-$1$ projectors to positive semidefinite matrices and at the same time provide significant improvements for the rank-$1$ case because we only require cubatures of strength $t\geq 3$ and not designs. For cubatures of strength $t\geq 4$, the rank-$1$ case is already covered in \cite{Kueng:2014rw}.

The overall structure of our proofs related to the reconstruction of symmetric rank-$1$ matricies from inner products with positive semidefinite matricies is guided by the approach in \cite{Gross:2013fk}. Our generalizations are based on the computation of trace moments of matrix distributions induced by the Haar measure on the orthogonal group. The use of zonal polynomials as discussed in \cite{James:1964mz,James:1974aa} enables us to compute all trace moments, and we explicitly provide the first $3$ of them. The remaining parts of the signal reconstruction proofs essentially follow the approach in \cite{Gross:2013fk} with adjusted parameters and constants combined with our results on the first three trace moments, cubatures \cite{Harpe:2005fk}, and random tight $t$-fusion frames \cite{Bachoc:2010aa,Ehler:2010aa}.

\subsection*{Outline} We introduce the general phase retrieval problem in Section \ref{sec:1}, where we also state the result from \cite{Bachoc:2012fk} about uniformly distributed rank-$k$-projectors. The findings in \cite{Gross:2013fk} for sampling spherical designs (hence the setting of rank-$1$ projectors) are stated in Section \ref{sec:3.2}. Deterministic conditions for signal recovery with positive semidefinite measurement matrices through solving the feasibility problem of a semidefinite program are verified in Section \ref{sec:deterministic} and are based on near isometry properties and a so-called approximate dual certificate. 

Our main result on phase retrieval is stated in Section \ref{sec:main and so} and is based on random cubatures.  
 The remaining part of the present paper is dedicated to its proof. The trace moments are computed in Section \ref{sec:trace}, special moments in Section \ref{sec:special} and the general case is treated in Section \ref{sec:general}. 
 We compute the first $3$ trace moments explicitly in Section \ref{sec:trace t=123}, which are an important ingredient of the proof of our main result on phase retrieval. Most of the technical details of the complete proof are contained in the appendix. Conclusions are given in Section \ref{sec:concl}.

\section{Phase retrieval and uniform sampling}\label{sec:1}
Let $\mathscr{H}_d$ denote the space of symmetric matrices in $\R^{d\times d}$ endowed with the Hilbert-Schmidt inner product $\langle X,Y\rangle:=\trace(XY^*)$, for $X,Y\in\mathscr{H}_d$. In our phase retrieval problem, we seek to recover some unknown signal $x\in\R^d$ from the knowledge of $n$ matrices $\{P_j\}_{j=1}^n\subset\mathscr{H}_d$ and the associated measurements
\begin{equation*}
\{\langle xx^*,P_{j}\rangle\}_{j=1}^n.
\end{equation*}
Clearly, $x$ can at best be recovered up to a global phase factor and so we simply aim to recover the rank-$1$ matrix $xx^*$. Uniqueness of $xx^*$ was discussed in \cite{Bodmann:2013fk,Balan:2013fk,Balan:2006fk,Bandeira:2013fk} for rank-$1$ orthogonal projectors $\{P_j\}_{j=1}^n$ and in \cite{Bachoc:2012fk,Cahill:2013fk} for more general choices of $\{P_j\}_{j=1}^n$.

Besides injectivity, we also need an efficient algorithm to eventually reconstruct the signal. We consider the set of measurement matrices 
\begin{equation*}
\mathcal{G}_{\lambda,d}:=\{ OD_\lambda O^*:O\in\mathcal{O}_d \},
\end{equation*}
where $D_\lambda=\diag(\lambda_1,\ldots,\lambda_d) $ and $\lambda=(\lambda_1,\ldots,\lambda_d)^*$ is a fixed vector with 
\begin{equation*}
1\geq \lambda_1\geq \ldots\geq \lambda_k>\lambda_{k+1}=\ldots,\lambda_d=0.
\end{equation*}
To derive asymptotic recovery results, we shall later increase the dimension $d$ while we keep $k$ and $\lambda_1,\ldots,\lambda_k$ fixed. Note that $\mathcal{G}_{\lambda,d}$ is simply the set of all rank-$k$ positive semidefinite matrices with nonzero eigenvalues $\lambda_1,\ldots,\lambda_k$. 


The Haar measure $dO$ on the orthogonal group $\mathcal{O}(d)$ acts transitively on $\mathcal{G}_{\lambda,d}$ by definition and induces a probability measure $\sigma_{\lambda,d}$ on $\mathcal{G}_{\lambda,d}$ that is invariant under the orthogonal group. When $\lambda_1=\cdots=\lambda_k=1$, then $\mathcal{G}_{\lambda,d}$ can be identified with the set of all $k$-dimensional linear subspaces in $\R^d$, known as the (real) Grassmann space. 

Choose $\{P_j\}_{j=1}^n\subset\mathcal{G}_{\lambda,d}$, and let us consider
\begin{equation}\label{eq:convex 0}
\text{find } X\in\mathscr{H}_d, \quad\text{subject to}\quad  \langle X,P_{j}\rangle = \langle xx^*,P_{j}\rangle,\; j=1,\ldots,n,\;\; X\succeq 0,
\end{equation}
where $X\succeq 0$ means that $X$ is positive semidefinite. This is the feasibility problem of a semidefinite program and efficient algorithms based on
interior point methods are available. For $\lambda=(1,0,\ldots,0)^*$, there is a constant $c>0$, such that the choice of $n\geq c d$ uniformly distributed subspaces yields that, with high probability, $xx^*$ is the only feasible point of \eqref{eq:convex 0}, cf.~\cite{Candes:2012fk,Candes:uq,Demanet:2012uq}. This result was generalized to rank-$k$ orthogonal projectors in \cite{Bachoc:2012fk}:
\begin{theorem}[\cite{Bachoc:2012fk}]\label{th:finale!}
Let $\lambda=(1,\ldots,1,0,\ldots,0)^*$, where $1$ is repeated $k$ times. Then there are constants $c_1,c_2>0$ such that, if
\begin{equation*}
n\geq c_1d
\end{equation*}
and $\{P_j\}_{j=1}^n\subset\mathcal{G}_{\lambda,d}$ are chosen independently identically distributed according to the normalized Haar measure $\sigma_{\lambda,d}$ on $\mathcal{G}_{\lambda,d}$,  then, for all $x\in\R^d$, the matrix $xx^*$ is the unique solution to \eqref{eq:convex 0} with probability at least $1-e^{-c_2 n}$.
\end{theorem}
For rank-$1$ projectors, i.e., $\lambda=e_1=(1,0,\ldots,0)^*$, the sampling from the uniform measure $\sigma_{\lambda,d}$ has been relaxed in \cite{Gross:2013fk} , which is the topic of the subsequent section.

\section{Signal reconstruction for rank-$1$ projectors}\label{sec:3.2}
This section deals with $\lambda=e_1$ only and before we cite some reconstruction results, we need to introduce further concepts and notation. 
For $x\in\mathbb{S}^{d-1}:=\{z\in\R^d:\|z\|=1\}$, we denote $P_x:=xx^*$. A collection $\{P_{x_j}\}_{j=1}^n\subset\mathcal{G}_{e_1,d}$ is called a \emph{projective $t$-design} if 
\begin{equation}\label{eq:def spherical design k=1}
\frac{1}{n}\sum_{j=1}^n \langle P_{x_j},P_x\rangle^{t} = \int_{\mathcal{G}_{e_1,d}}\langle P,P_x\rangle^{t}d\sigma_{e_1,d}(P)  ,\quad\text{for all } x\in\mathbb{S}^{d-1}.
\end{equation}
The latter is equivalent to 
\begin{equation*}
\frac{1}{n}\sum_{j=1}^n |\langle x_j,x\rangle|^{2t} = \int_{\mathbb{S}^{d-1}}|\langle y,x\rangle|^{2t}dy  ,\quad\text{for all } x\in\mathbb{S}^{d-1},
\end{equation*}
where $dy$ denotes, as usual, the canonical measure on the sphere that we additionally assume to be normalized. 

 The reconstruction results in \cite{Gross:2013fk} were only derived for complex signals and measurements, but can also be checked in the real case by analoguous arguments. For consistency with the presentation here, we shall therefore recall this result in the real setting. So, for $x\in\R^d$ and $\{P_j\}_{j=1}^n\subset\mathcal{G}_{e_1,d}$, we consider the optimization problem
\begin{equation}\label{rank 1 opt}
\arg\!\!\!\min_{X\in\mathscr{H}_d,\;X\succeq 0} \!\|X\|_*,\quad \text{ s.t. } \trace(X)=\|x\|^2,\;\big(\langle P_{j},X\rangle\big)_{j=1}^n = \big(\langle P_{j},xx^*\rangle\big)_{j=1}^n,
\end{equation}
where $\|X\|_*$ denotes the nuclear norm of $X$, i.e., the sum of the absolute values of its singular values, then \cite{Gross:2013fk} yields:
\begin{theorem}[\cite{Gross:2013fk}]\label{theorem:Gross}
Let $x\in\R^d$ be an unknown signal and suppose that $\|x\|^2$ is known. If $\{P_j\}_{j=1}^n\subset\mathcal{G}_{e_1,d}$ is independently sampled in a uniform fashion from some projective design of strength $t\geq 3$, then with probability at least $1-e^{-\omega}$, the rank-one matrix $xx^*$ is the unique solution to \eqref{rank 1 opt} provided that
\begin{equation*}
n\geq c_1\omega t d^{1+2/t}\log^2(d),
\end{equation*}
where $\omega\geq 1$ is an arbitrary parameter and $c_1$ is a universal constant.
\end{theorem}
Note that the above Theorem \ref{theorem:Gross} is restricted to uniform sampling of a projective $t$-design. The latter is a rather inconvenient restriction as shown in the following example:
\begin{example}\label{EX:1}
The classical phase retrieval problem stemming from optical physics involves Fourier measurements, meaning that the rank-$1$ projectors $P_j=x_jx_j^*$ are generated by Fourier vectors
\begin{equation}\label{eq:FT special}
x_j = \frac{1}{\sqrt{d}}(-e^{2\pi \mathrm{i} l_1 j/m},\ldots,e^{-2\pi \mathrm{i} l_d j/m})^*\in\C^d,
\end{equation}
where $\{l_i\}_{i=1}^d\subset \Z$. Often, magnitude measurements in time are also available, expressed as additional measurements $\{P_{e_k}\}_{k=1}^n$, where $\{e_k\}_{k=1}^d$ is the canonical orthogonal basis of $\C^d$.  It turns out that the combination of special Fourier vectors with time measurements yield a formula similar to \eqref{eq:def spherical design k=1}, hence, yields almost a projective design. In fact, these ideas are inspired by \cite[Proposition 4]{Konig:1999fk} and \cite[Section 2.1.2]{Strohmer:2003aa}, see also \cite{Ehler:2013fk1}: let $q$ be a prime and let $d=q^r+1$ for some $r\in\N$. For $m=d^2-d+1$, there exist integers $0\leq l_1<\cdots <l_d<m$ such that all numbers $1,\ldots,m-1$ occur as residues $\!\!\!\!\mod\,m$ of the $d(d-1)$ differences $(l_k-l_{\ell})$, for $k\neq \ell$, cf.~\cite{Konig:1999fk}. Then the following formula holds, for all $x\in\C^d$ with $\|x\|=1$, 
\begin{equation*}
\frac{d}{d^3+1}\sum_{j=1}^{d^2-d+1} \langle P_{x_j},P_x\rangle^{2} + \frac{1}{d(d+1)} \sum_{k=1}^{d} \langle P_{e_k},P_x\rangle^{2}= \int_{\mathcal{G}^\C_{e_1,d}}\langle P,P_x\rangle^{2}d\sigma^\C_{e_1,d}(P),
\end{equation*}
where $\mathcal{G}^\C_{e_1,d}$ denotes the complex projective space and $\sigma^\C_{e_1,d}$ its normalized canonical measure induced by the Haar measure on the unitary group. Thus, the combined Fourier and time measurements provide some sort of weighted projective design.
\end{example}

\begin{remark}
Although our presentation is focused on the real case, we want to point out that all results can be derived in the complex setting as well, so that Example \ref{EX:1} can still guide us. It shows that the structural requirement of a design in Theorem \ref{theorem:Gross} is still too restrictive. We shall generalize Theorem \ref{theorem:Gross} in several aspects. First, it is restricted to $\lambda=e_1$, and we shall address the general case $\lambda$. Moreover, we can handle weighted designs, which is a significant structural generalization, so that our results also yield significant improvements for $\lambda=e_1$.

\end{remark}



\section{Deterministic conditions for signal reconstruction with general $\lambda$}\label{sec:deterministic}
This section is dedicated to consider phase retrieval when $\lambda$ is arbitrary. 
To model the knowledge of $\|x\|^2$, we make the convention that $P_{0}=I_d$ and, hence, $\langle xx^*,P_{0}\rangle = \trace(xx^*)=\|x\|^2$ holds, and we consider the problem
\begin{equation}\label{eq:optimization}
 \text{find }X\in\mathscr{H}_d, \quad \text{such that}\quad  \big(\langle X,P_{j}\rangle\big)_{j=0}^n = \big(\|P_{j}x\|^2\big)_{j=0}^n,\quad X\succeq 0.
\end{equation}
Note that \eqref{eq:optimization} is the feasibility problem of a semidefinite programm. In comparison to \eqref{rank 1 opt}, the actual minimization is void. In fact, $X\succeq 0$ yields $\|X\|_*=\trace(X)$, so that the minimization in \eqref{rank 1 opt} was superfluous too.


To establish deterministic conditions that ensure solvability of \eqref{eq:optimization}, we use the notion of dual certificates that require some preparation. For a fixed $x\in\R^d$, we consider the subspace
\begin{equation*}
T_x:=\{xz^*+zx^* : z\in\R^d \}\subset\mathscr{H}_d,
\end{equation*}
which is the tangent space of the rank-one symmetric matrices at the point $xx^*$. 
For some $Y\in\mathscr{H}_d$, let $Y_{T_x}$ denote the orthogonal projection of $Y$ onto $T_x$ and $Y_{T^\perp_x}$ the orthogonal projection onto the orthogonal complement of $T_x$. Moreover, let $\|\cdot\|_{F}$ denote the Frobenius norm and $\|\cdot\|_{Op}$ the spectral norm.
\begin{definition}
For $\{P_j\}_{j=1}^n\subset\mathcal{G}_{\lambda,d}$, we call $Y\in\mathscr{H}_d$ a \emph{$(\gamma,\delta)$-dual certificate} with respect to $x\in\R^d$ if $Y\in\spann\{I,P_{1},\ldots,P_{n}  \}$ and
\begin{equation}\label{eq:dual cert}
\|Y_{T_x} - xx^*\|_{F} \leq \gamma\quad\text{and}\quad \|Y_{T^\perp_x}\|_{Op} \leq \delta.
\end{equation}
\end{definition}
For notational convenience, we introduce the mapping
\begin{equation}\label{eq:A}
\mathcal{A}_n : \mathscr{H}_d\rightarrow \R^n,\quad X\mapsto \big(\langle X,P_{j}\rangle \big)_{j=1}^n,
\end{equation}
for $\{P_j\}_{j=1}^n\subset\mathcal{G}_{\lambda,d}$.

Now, we can formulate deterministic recovery guarantees:
\begin{theorem}\label{th:fundamental deterministic}
Suppose that there are $\alpha,\beta>0$ and $\{P_j\}_{j=1}^n\subset\mathcal{G}_{\lambda,d}$ satisfying
\begin{equation}\label{eq:low ineq0}
\alpha\|X\|_{F}^2 \leq\frac{1}{n}\|\mathcal{A}_n(X)\|^2  \leq \beta\|X\|_{F}^2,
\end{equation}
where $\mathcal{A}_n$ is given by \eqref{eq:A} and the lower inequality holds for all matrices $0\neq X\in T_x$, and the upper one for all $X\in \mathscr{H}_d$.
If a $(\gamma,\delta)$-dual certificate $Y$ with respect to $x$ exists and
\begin{equation*}
\sqrt{\frac{\beta}{\alpha}} < \frac{1-\delta}{\gamma},
\end{equation*}
then $xx^*$ is the unique solution to \eqref{eq:optimization}.
\end{theorem}
\begin{proof}
We know that $xx^*$ solves \eqref{eq:optimization}.
Suppose that $X$ is another solution and put $\Delta:=X-xx^*$. As in \cite{Gross:2013fk}, we apply the pinching inequality, cf.~\cite{Gross:2013fk,Bhatia:1996fk}, to obtain
\begin{equation*}
\trace(X) = \trace(xx^*+\Delta) \geq \trace(xx^*) + \trace(\Delta_{T_x})+ \|\Delta_{T^\perp_x}\|_*.
\end{equation*}
Since $\trace(X)=\trace(xx^*)=\|x\|^2$, we obtain
\begin{equation}\label{eq:contra}
0\geq \trace(\Delta_{T_x})+\|\Delta_{T^\perp_x}\|_*.
\end{equation}
If $\Delta_{T_x}=0$, then  we derive $\Delta_{T^\perp_x}=0$, so that $\Delta=0$ and hence $X=xx^*$. If $\Delta_{T_x}\neq 0$, then \eqref{eq:low ineq0} implies
\begin{equation}\label{eq:est in proof}
\|\Delta_{T_x}\|_{F} \leq\sqrt{\frac{1}{\alpha n}} \|\mathcal{A}_n(\Delta_{T_x})\| = \sqrt{\frac{1}{\alpha n}} \|\mathcal{A}_n(\Delta_{T^\perp_x})\|\leq \sqrt{\frac{\beta}{\alpha }} \|\Delta_{T^\perp_x}\|_{F}.
\end{equation}
Next, we observe that $ \langle xx^*,\Delta_{T_x}\rangle=\trace(\Delta_{T_x})$ and obtain
\begin{align*}
0 &= \langle Y,\Delta\rangle  =  \langle Y_{T_x}-xx^*,\Delta_{T_x}\rangle  +  \langle xx^*,\Delta_{T_x}\rangle  + \langle Y_{T^\perp_x} ,\Delta_{T^\perp_x}\rangle  \\
& \leq \|Y_{T_x}-xx^*\|_{F}\|\Delta_{T_x}\|_{F} +\trace(\Delta_{T_x})+ \|Y_{T^\perp_x}\|_{Op} \|\Delta_{T^\perp_x}\|_*\\
& \leq \trace(\Delta_{T_x})+\|Y_{T_x}-xx^*\|_{F}\sqrt{\frac{\beta}{\alpha }} \|\Delta_{T^\perp_x}\|_{F}+\delta\|\Delta_{T^\perp_x}\|_*\\
&\leq \trace(\Delta_{T_x})+\gamma \sqrt{\frac{\beta}{\alpha }} \|\Delta_{T^\perp_x}\|_{F}+\delta\|\Delta_{T^\perp_x}\|_*\\
& \leq \trace(\Delta_{T_x})+\big(\gamma \sqrt{\frac{\beta}{\alpha }}+\delta\big)\|\Delta_{T^\perp_x}\|_*\\
& \leq \trace(\Delta_{T_x})+\|\Delta_{T^\perp_x}\|_*.
\end{align*}
Since $\Delta_{T_x}\neq 0$, the inequalities \eqref{eq:est in proof} yield $\Delta_{T^\perp_x}\neq 0$, so that the inequality of the last line is strict, which is a contradiction to \eqref{eq:contra}. Therefore, we have $\Delta=0$ and hence $X=xx^*$, so that $xx^*$ is the unique solution to \eqref{eq:optimization}.
\end{proof}

\section{Cubatures for phase retrieval with general $\lambda$}\label{sec:main and so}
We aim to verify that certain random samples in $\mathcal{G}_{\lambda,d}$ satisfy the conditions of Theorem \ref{th:fundamental deterministic} with a decent probability, so that signal recovery is guaranteed. To characterize the type of random distributions involved, we need to define some sort of weighted design on $\mathcal{G}_{\lambda,d}$, for which we shall first introduce trace moments:
\begin{definition}
The \emph{$t$-th trace moments} (or trace moments of degree $t$) of some random matrix $\mathcal{P}\in\mathcal{G}_{\lambda,d}$ 
are
\begin{equation*}
\mu^t_\mathcal{P}(X):=
\EE\big(\langle \mathcal{P}, X \rangle^t\big)
,\quad X\in \mathscr{H}_d.
\end{equation*}
The trace moments of $\mathcal{P}$ distributed according to $\sigma_{\lambda,d}$ are denoted by 
\begin{equation*}
\mu^t_{\lambda,d}(X):=\int_{\mathcal{G}_{\lambda,d}}\langle P, X \rangle^t d\sigma_{\lambda,d}(P).
\end{equation*}
Similarly, for $\beta\in\N^s$, we define cross-moments by
\begin{equation*}
\mu^\beta_{\mathcal{P}}(X_1,\ldots,X_s)=\mathbb{E}\big(\langle X_1,\mathcal{P}\rangle^{\beta_1}\cdots \langle X_s,\mathcal{P}\rangle^{\beta_s}\big),\quad X_1,\ldots,X_s\in\mathscr{H}_d
\end{equation*}
and make use of the expression $\mu^\beta_{\lambda,d}(X_1,\ldots,X_s)$, respectively. If $\beta$ consists of ones only, then we simply write $\mu_{\mathcal{P}}(X_1,\ldots,X_s)$ and $\mu_{\lambda,d}(X_1,\ldots,X_s)$. 
\end{definition}
Next, we can introduce cubatures for $\mathcal{G}_{\lambda,d}$:
\begin{definition}
A $\mathcal{G}_{\lambda,d}$-valued random variable $\mathcal{P}$ is called a \emph{random cubature} of strength $t$ (in $\mathcal{G}_{\lambda,d}$) if its $t$-th trace moments coincide with those of $\sigma_{\lambda,d}$, i.e., 
\begin{equation}\label{eq:def cub}
\mu^t_\mathcal{P}(X) = \mu^t_{\lambda,d}(X),\quad\text{for all } X\in\mathscr{H}_d.
\end{equation}
If $\mathcal{P}$ satisfies \eqref{eq:def cub} at least for all $X=xx^*$, $x\in\R^d$, then it is called a \emph{random tight $t$-fusion frame}.  
\end{definition}
\begin{remark}
In the literature, the term tight $t$-fusion frame usually refers to the case when the entries in $\lambda$ are ones and zeros, so that the measurement matrices are orthogonal projectors. Here, we use this term in a slightly more general sense.
\end{remark}

If $\lambda=e_1$ holds, then any random tight $t$-fusion frame is already a random cubature of strength $t$. Still for $\lambda=e_1$, let us consider a random cubature $\mathcal{P}\in\mathcal{G}_{e_1,d}$ with finite support, say $\{P_j\}_{j=1}^n$, and corresponding weight distribution $\{\omega_j\}_{j=1}^n$. Then strength $t$ implies that 
\begin{align*}
\sum_{j=1}^n \omega_j \langle P_{j},P_{x}\rangle^t=\mathbb{E}(\langle \mathcal{P},P_x\rangle^t) &=\mu^t_\mathcal{P}(P_x)\\ &=\mu^t_{e_1,d}(P_x)=\int_{\mathcal{G}_{e_1,d}}  \langle P_{V},P_{x}\rangle^t d\sigma_{e_1,d}(V)
\end{align*}
holds for all $P_x$, which becomes formula \eqref{eq:def spherical design k=1} when the weights are constant. Thus, (random) cubatures are a more flexible concept than designs.

The trace moments as functions on $\mathcal{G}_{\lambda,d}$ generate polynomial function spaces, and we define
\begin{equation}\label{eq:polys}
\Pol_{t} (\mathcal{G}_{\lambda,d}):=\spann \{ \langle \cdot,X_1\rangle\cdots\langle \cdot, X_t\rangle\big|_{\mathcal{G}_{\lambda,d}} : X_1,\ldots,X_t\in \mathscr{H}_d \}.
\end{equation}
We also define the subspace
\begin{equation}\label{eq:polys diag}
\Pol^1_{t} (\mathcal{G}_{\lambda,d}):=\spann \{ \langle \cdot,P_x\rangle ^t \big|_{\mathcal{G}_{\lambda,d}} : x\in \mathbb{S}^{d-1}\}.
\end{equation}
Existence of cubatures is quite well-understood, and the following results are based on findings in \cite{Harpe:2005fk}. In fact, the second part of the following proposition is completely contained in \cite{Harpe:2005fk}. The first part is an analogous proof, cf.~\cite{Ehler:2014zl}:
\begin{proposition}\label{prop:exist}
There exists a random tight $t$-fusion frame $\mathcal{P}\in\mathcal{G}_{\lambda,d}$ distributed according to some probability measure $\nu$ such that
\begin{equation*}
\#\supp(\nu)\leq \dim(\Pol_{t}^1(\mathcal{G}_{\lambda,d}))+1.
\end{equation*}
 There exists a random cubature $\mathcal{P}\in\mathcal{G}_{\lambda,d}$ of strength $t$ distributed according to some probability measure $\nu$ such that
 \begin{equation*}
 \#\supp(\nu)\leq \dim(\Pol_{t}(\mathcal{G}_{\lambda,d}))-1.
 \end{equation*}
\end{proposition}
It is noteworthy that the dimension of $\Pol_{t}(\mathcal{G}_{\lambda,d})$ can be bounded by the number of monomials of degree $t$ in $\frac{1}{2}d(d+1)$ variables, i.e., 
\begin{equation*}
\dim(\Pol_{t}^1(\mathcal{G}_{\lambda,d}))\leq \dim(\Pol_{t}(\mathcal{G}_{\lambda,d}))\leq \binom{\frac{1}{2}d(d+1)+t-1}{t}.
\end{equation*}
There is also a tighter bound for $\dim(\Pol_{t}^1(\mathcal{G}_{\lambda,d}))$, i.e., 
\begin{equation}\label{eq:hom 1}
\dim(\Pol_{t}^1(\mathcal{G}_{\lambda,d}))\leq \binom{d+2t-1}{2t},
\end{equation}
which is a consequence of the following result showing that the dimension can be bounded by the dimension of the homogeneous polynomials of degree $2t$ in $d$ variables.
\begin{lemma}
For $t\in\N$, we obtain
\begin{equation*}
\dim(\Pol_t^1(\mathcal{G}_{\lambda,d}))=\dim(\spann\{\|P\cdot \|^{2t}\big|_{\mathbb{S}^{d-1}} : P\in \mathcal{G}_{\sqrt{\lambda},d}\}),
\end{equation*}
where $\sqrt{\lambda}=(\sqrt{\lambda_1},\ldots,\sqrt{\lambda_d})^*$.
\end{lemma}
\begin{proof}
Let $\{x_i\}_{i=1}^n\subset S^{d-1}$ be such that $\langle x_ix_i^*,\cdot\rangle^t\big|_{\mathcal{G}_{\lambda,d}}$, $i=1,\ldots,n$ are linearly independent. By classical arguments, there are $\{P_j\}_{j=1}^n\subset \mathcal{G}_{\lambda,d}$, such that the matrix $(\langle x_ix_i^*,P_j\rangle^t)_{i,j}$ is invertible. Therefore, the functions $\|P^{1/2}_j\cdot \|^{2t}\big|_{S^{d-1}}$, $j=1,\ldots,n$, are linearly independent since $\|P^{1/2}_j x \|^{2t}=\langle xx^*,P_j\rangle^t$, for all $x\in S^{d-1}$. The same arguments apply vice versa, which concludes the proof.
\end{proof}

It should also be noted that existence of cubatures on the sphere when their support is fixed, i.e., designing the mass distribution according to some fixed locations, have been investigated in \cite{Mhaskar:2002ys}, and we refer to \cite{Filbir:2010aa} for more general manifolds. However, general existence results for designs with specific bounds similar to Proposition \ref{prop:exist} are not known.

After having established existence of cubatures, we can now state our main result on phase retrieval, which generalizes Theorem \ref{theorem:Gross}.
\begin{theorem}\label{theorem:general ehler}
Suppose that $\|x\|^2$ is known and that $\{\mathcal{P}_j\}_{j=1}^n\subset\mathcal{G}_{\lambda,d}$ are independently sampled from a random cubature of strength $3$, which is also a random tight $t$-fusion frame for some $t\geq 3$. Then with probability at least $1-e^{-\omega}$, the rank-one matrix $xx^*$ is the unique solution to \eqref{eq:optimization} provided that
\begin{equation}\label{eq:number}
n\geq c_1\omega  t d^{1+2/t}\log^2(d),
\end{equation}
where $\omega\geq 1$ is an arbitrary parameter and $c_1$ is a constant, which does not depend on $d$.
\end{theorem}
Few comments are in order. In contrast to Theorem \ref{theorem:Gross}, we allow random cubatures that are not uniformly distributed on their support. Furthermore, we can separate the cubature condition of strength $3$ from the tight frame requirements for $t > 3$, which are indeed different concepts when $k$ is bigger than $1$. Note that the number of measurements $n$ scales linearly in the ambient dimension $d$ up to logarithmic factors if we choose $t=\log(d)$ because then  \eqref{eq:number} yields $n\geq c_1 \omega d \log^3(d)$.

\section{General structure of the proof of Theorem \ref{theorem:general ehler}}
The proof of Theorem \ref{theorem:general ehler} is guided by the structure provided in \cite{Gross:2013fk} and based on the following two results about near isometry properties and the existence of dual certificates as required by Theorem \ref{th:fundamental deterministic}.
\begin{theorem}\label{theorem:lower inequality}
If $\{P_j\}_{j=1}^n\subset\mathcal{G}_{\lambda,d}$ are independent and identical copies of a random matrix $\mathcal{P}\in\mathcal{G}_{\lambda,d}$ that is a random cubature of strength $3$, then, for any sufficiently large constant $C_0$, there is a constant $c>0$ such that
\begin{equation}\label{eq:low ineq}
\frac{1}{C_0d^2}\|X\|_{F}^2 \leq \frac{1}{n}\|\mathcal{A}_n(X)\|^2
\end{equation}
holds for all matrices $X\in T_x$ simultaneously with probability of failure at most $d^2e^{-c\frac{n}{d}}$.
\end{theorem}
Note that the constant $C_0$ will be used in the remaining part of the present paper. 
We still need an approximate dual certificate though.
\begin{theorem}\label{th:dual cert}
Suppose that $0\neq x\in\R^d$, that $\omega\geq 1$, and that $\mathcal{P}\in\mathcal{G}_{\lambda,d}$ is a random cubature of strength $3$ and a random tight $t$-fusion frame for some $t\geq 3$. Then, for any sufficiently large constant $c_0$, there is a constant $c>0$ such that if the number of measurements satisfies
\begin{equation}\label{eq:final the}
n\geq c \omega t d^{1+2/t} \log^2(d),
\end{equation}
then with probability of failure at most $\frac{1}{2}e^{-\omega}$, there exists a $(\frac{1}{c_0 d},\frac{1}{c_0})$-dual certificate with respect to $x$.
\end{theorem}
Note that the constant $c_0$ is used in the remaining part of the present paper. Now, we have all ingredients for the proof of our main result on phase retrieval:
\begin{proof}[Proof of Theorem \ref{theorem:general ehler}]
Guided by the structure provided in \cite{Gross:2013fk}, we aim to apply Theorem \ref{th:fundamental deterministic}, and the upper bound in the near isometry property can easily be verified. Indeed, for any collection $\{P_j\}_{j=1}^n\subset\mathcal{G}_{\lambda,d}$, the Cauchy-Schwartz inequality and the assumptions on $\lambda$ yield
\begin{equation*}
\frac{1}{n}\|\mathcal{A}_n(X)\|^2  = \frac{1}{n}\sum_{j=1}^n \langle X,P_{j}\rangle ^2
 \leq \frac{1}{n}\sum_{j=1}^n \| X\|^2_{F}\|P_{j}\|^2_{F}
 \leq k \| X\|^2_{F},\quad \text{for all }X\in \mathscr{H}_d.
\end{equation*}
Hence, we can choose $\beta:=k$. According to Theorem \ref{theorem:lower inequality}, we can select $\alpha = \frac{1}{C_0d^2}$. By choosing $\gamma:=\frac{1}{c_0 d}$ and $\delta:=\frac{1}{c_0}$ with $c_0 >\sqrt{kC_0}+1$, Theorem \ref{th:dual cert} yields a $(\gamma,\delta)$-dual certificate for the required number of measurements, and we have the estimate
\begin{equation*}
\sqrt{\frac{\beta}{\alpha}}= \sqrt{C_0}d\sqrt{k}< d(c_0-1)=\frac{1-\delta}{\gamma}.
\end{equation*}
Thus, the assumptions in Theorem \ref{th:fundamental deterministic} are satisfied. The corresponding probabilities work out nicely by applying $d^2e^{-c\frac{n}{d}}=\frac{1}{2}e^{\log(2)+2\log(d)-c\frac{n}{d}}$ and $\omega\geq 1$, which concludes the proof. 
\end{proof}
In order to complete the proof of Theorem \ref{theorem:general ehler}, we must still verify the Theorems \ref{theorem:lower inequality} and \ref{th:dual cert}. Their proofs require the actual computation of $t$-th trace moments of $\mathcal{P}\sim\sigma_{\lambda,d}$ for $t=1,2,3$ that we shall discuss in the subsequent sections. In fact, we shall present a closed formula for the $t$-th trace moments for all $t$ based on zonal polynomials.

%
%

\section{Computing trace moments}\label{sec:trace}

\subsection{Some special trace moments}\label{sec:special}
For special choices of $\lambda$ and $X$, the trace moments of $\mathcal{P}\sim\sigma_{\lambda,d}$ are already known. If $\lambda=(1,\dots,1,0,\ldots,0)^*$, where $1$ is repeated $k$ times, an explicit expression for the moments of rank-$1$ matrices $X=xx^*$ can be derived,
\begin{equation}\label{fusion frame bounds}
\mu^t_{\lambda,d}(xx^*) = \frac{(k/2)_t}{(d/2)_t}\cdot \|x\|^{2t},\quad\text{for all } x\in\R^d,
\end{equation}
where $(a)_t:=a(a+t)\cdots (a+t-1)$ denotes the Pochhammer symbol, cf.~\cite{Bachoc:2010aa}. Recall that those are the moments needed for the characterization of random tight $t$-fusion frames. 

Moreover, we further restrict $\lambda$ to derive explicit formulas for more general moments. If $\mathcal{P}\sim\sigma_{e_1,d}$ and $\{x_i\}_{i=1}^d$ is an orthonormal basis for $\R^d$, then one can verify that the vector $\big(\langle \mathcal{P},x_1x_1^*\rangle,\ldots,\langle \mathcal{P},x_dx_d^*\rangle  \big)^*$ is Dirichlet distributed with parameter vector $(1/2,\ldots,1/2)$, \cite{Barthe:2010fc}. The generalized moments of such Dirichlet distributed random vectors are known \cite{Wong:1998cr}, and, indeed, if $\beta\in\N^d $, we obtain
\begin{equation}\label{eq:mixed orth}
\mu^\beta_{e_1,d}(x_1x_1^*,\ldots,x_dx_d^*)= \frac{\prod_{i=1}^d (1/2)_{\beta_i}}{(d/2)_{|\beta|}},
\end{equation}
where $|\beta|=\sum_{i=1}^d \beta_i$. 
Since $\sigma_{e_1,d}$ is invariant under the orthogonal group, the terms in \eqref{eq:mixed orth} do not depend on the special choice of the orthonormal basis.
Hence, the spectral decomposition of $X\in\mathscr{H}_d$ yields a closed formula,
\begin{equation}
\mu^t_{e_1,d}(X) = \sum_{\substack{\beta\in\N^d\\|\beta|= t}} \binom{t}{\beta}  \alpha^\beta  \frac{\prod_{i=1}^d (1/2)_{\beta_i}} {(d/2)_{t}},\quad\text{for } X\in\mathscr{H}_d,
\end{equation}
where $\alpha=(\alpha_1,\ldots,\alpha_d)$ are the eigenvalues of $X$.

\subsection{Trace moments for general $\lambda$, $t$, and $X$}\label{sec:general}
Computing trace moments when $\lambda$ is more general requires the theory of zonal polynomials as developed in \cite{James:1964mz,James:1974aa}, see also the textbooks \cite{Chikuse:2003aa,Muirhead:1982fk}. Zonal polynomials are homogeneous polynomials in $\mathscr{H}_d$, which are invariant under conjugation with respect to the orthogonal group. According to representation theory, those polynomials $C_\pi$ are indexed by all partitions $\pi$ of the natural numbers. Here, a partition of $t$ are integer vectors $\pi=(\pi_1,\dots,\pi_t)$ with $\pi_1\geq \ldots\geq \pi_t\geq 0$ and $\sum_{i=1}^t\pi_i=t$. The number of parts of $\pi$ is the number of nonzero entries. The set of partitions of $t$ with no more than $d$ parts is denoted by $\mathscr{P}_{t,d}$.  

To compute cross-moments of a random matrix $\mathcal{P}\in\mathcal{G}_{\lambda,d}$, we shall make use of the following combinatorial fact:
\begin{lemma}\label{lemma:direct}
For any integer $t\geq 1$ and $x_1\ldots,x_t\in\R$, we have
\begin{equation*}
t! x_1\cdots x_t = \sum_{ J\subset\{1,\ldots,t\}}(-1)^{t+\# J}\big(\sum_{j\in J} x_j  \big)^t.
\end{equation*}
\end{lemma}
\begin{proof}
Consider the homogeneous and symmetric polynomials 
\begin{equation*}
S^\ell_t(x_1,\ldots,x_t):=\sum_{\substack{J\subset\{1,\ldots,t\} \\ \#J=\ell}}  \big( \sum_{j\in J} x_j\big)^t
\end{equation*}
of degree $t$. 
The coefficient of the monomial $x^\beta$, for $\beta\in\N^t$, $|\beta|=t$, in $\big( \sum_{j\in J} x_j\big)^t$ is 
\begin{equation*}
\begin{cases} \binom{t}{\beta},& \supp(\beta)\subset J,\\ 0,& \text{otherwise},
\end{cases}
\end{equation*}
where $\binom{t}{\beta}=\frac{t!}{\beta_1!\cdots \beta_t!}$. Together with   
\begin{align*}
 \#\{J\subset\{1,\ldots,t\} : \#J=\ell,\; \supp(\beta)\subset J   \} & 
= \binom{t-\#\supp(\beta)}{\ell-\#\supp(\beta)},
\end{align*}
we can conclude
\begin{equation*}
S^\ell_t(x_1,\ldots,x_t) = \sum_{\substack{\beta\in\N^t\\|\beta|=t}} \binom{t}{\beta}\binom{t-\#\supp(\beta)}{\ell-\#\supp(\beta)} x^\beta.
\end{equation*}
This yields
\begin{align*}
\sum_{ J\subset\{1,\ldots,t\}}(-1)^{t+\# J}\big(\sum_{j\in J} x_j  \big)^t & = \sum_{\substack{\beta\in\N^t\\|\beta|=t}}\binom{t}{\beta}x^\beta\sum_{\ell=\#\supp(\beta)}^t (-1)^{t+\ell}\binom{t-\#\supp(\beta)}{t-\ell}\\
& = \sum_{\substack{\beta\in\N^t\\|\beta|=t}}\binom{t}{\beta}x^\beta (-1)^{t+\#\supp(\beta)}\sum_{\ell=0}^{t-\#\supp(\beta)} (-1)^{\ell}\binom{t-\#\supp(\beta)}{\ell}\\
& = \sum_{\substack{\beta\in\N^t\\|\beta|=t}}\binom{t}{\beta}x^\beta (-1)^{t+\#\supp(\beta)}(1-1)^{t-\#\supp(\beta)},
\end{align*}
with $0^0:=1$, which concludes the proof. 
\end{proof}

The latter lemma enables us to actually compute trace moments:
\begin{theorem}\label{th:and so}
The cross-moments of any random matrix $\mathcal{P}\in\mathcal{G}_{\lambda,d}$ satisfy 
\begin{equation}\label{eq:crosser}
\mu_{\mathcal{P}}(X_1,\ldots,X_t) = \frac{1}{t!}\sum_{J\subset\{1,\ldots,t\} }(-1)^{t+\# J}  \mu^t_{\mathcal{P}}(\sum_{j\in J} X_j),\quad X_1,\ldots,X_t\in\mathscr{H}_d.
\end{equation}
In particular, if $\mathcal{P}\sim\sigma_{\lambda,d}$, then \eqref{eq:crosser} can be computed by
\begin{equation}\label{eq:123}
\mu^t_{\lambda,d}(X) = \sum_{\pi\in\mathscr{P}_{t,d}}\frac{C_\pi(X)C_\pi(D_\lambda)}{C_\pi(I_d)},\quad\text{for all } X\in\mathscr{H}_d,
\end{equation}
where $D_\lambda=\diag(\lambda_1,\ldots,\lambda_d)^*$.
\end{theorem}
\begin{proof}
The formula \eqref{eq:crosser} is a direct consequence of Lemma \ref{lemma:direct}. 

Equation \eqref{eq:123} follows from properties of zonal polynomials, cf.~\cite{James:1964mz}, namely
\begin{align*}
\trace(X)^t& = \sum_{\pi\in \mathscr{P}_t} C_\pi(X),\\
\int_{\mathcal{O}(d)} C_\pi (XOYO^*) dO &= \frac{C_\pi(X)C_\pi(Y)}{C_\pi(I_d)},
\end{align*}
for all $X,Y\in\mathscr{H}_d$, and that $\sigma_{\lambda,d}$ is induced by the Haar measure on the orthogonal group implying 
\begin{equation*}
\int_{\mathcal{G}_{\lambda,d}} C_\pi (XP) d\sigma_{\lambda,d}(P)=\int_{\mathcal{O}(d)} C_\pi (XOD_\lambda O^*) dO.\qedhere
\end{equation*}
\end{proof}

\subsection{Explicit trace moments for $t=1,2,3$}\label{sec:trace t=123}
To make use of Theorem \ref{th:and so} enabling us to compute trace moments of $\sigma_{\lambda,d}$ for $t=1,2,3$, we still need explicit forms of the zonal polynomials. Indeed, they were computed in \cite{James:1964mz}:
\begin{align*}
C_{(1)}(X) & = \trace(X)\\
C_{(2)}(X) & = \frac{1}{3}\big(\trace^2(X)+2\trace(X^2)\big)\\
C_{(1,1)}(X) & = \frac{2}{3}\big(\trace^2(X) - \trace(X^2)\big)\\
C_{(3)}(X) & = \frac{1}{15}\big(\trace^3(X) +6\trace(X)\trace(X^2)+8\trace(X^3)\big) \\
C_{(2,1)}(X) & =\frac{3}{5}\big( \trace^3(X) +\trace(X)\trace(X^2)-2\trace(X^3)\big)\\
C_{(1,1,1)}(X) & =\frac{1}{3}\big( \trace^3(X) -3\trace(X)\trace(X^2)+2\trace(X^3)	\big)
\end{align*}
%
%
%

We can now apply Theorem \ref{th:and so}, which yields the trace moments for $t=1,2,3$:
\begin{theorem}\label{th:all contained}
For all $d\geq 3$ and  $X_1,X_2,X_3\in\mathscr{H}_d$, we have
\begin{align*}
\mu_{\lambda,d}(X_1) &=\frac{1}{q_{1,d}}\alpha_{(1)}\trace(X_1),\\
\mu_{\lambda,d}(X_1,X_2)&= \frac{1}{q_{2,d}}\big(\alpha_{(1,1)}\trace(X_1)\trace(X_2)+\alpha_{(2)}\trace(X_1X_2)\big),\\
\mu_{\lambda,d}(X_1,X_2,X_3)&=\frac{1}{q_{3,d}}\big( \alpha_{(1,1,1)}\trace(X_1)\trace(X_2)\trace(X_3)+\\
&\quad \frac{\alpha_{(2,1)}}{3}(\trace(X_1)\trace(X_2X_3)+\trace(X_2)\trace(X_1X_3)+\trace(X_3)\trace(X_1X_2))\\
&\quad \alpha_{(3)} \trace(X_1X_2X_3) \big),
\end{align*}
where we set $s_i:=\trace(D_\lambda^i)$ and 
\begin{align*}
q_{1,d}&=d,\\
\alpha_{(1)}&=s_1,
\end{align*}
\begin{align*}
q_{2,d}&=(d-1)d(d+2),\\
\alpha_{(1,1)} &=(d+1)s_1^2-2s_2\\
\alpha_{(2)} &= -2s_1^2+2d s_2,
\end{align*}
\begin{align*}
q_{3,d}&= (d-2)(d-1)d(d+2)(d+4),\\
\alpha_{(1,1,1)}& =(d^2+3d-2)s_1^3-6(d+2)s_1s_2+16s_3,\\
\alpha_{(2,1)}& =-6(d+2)s_1^3  +  6(d^2+2d+4) s_1s_2-24ds_3 ,\\
\alpha_{(3)}& = 16s^3_1-24ds_1s_2+8d^2s_3.
\end{align*}
\end{theorem}

If we keep the last matrix argument undetermined, then we derive the following result, which is simply a weak formulation of Theorem \ref{th:all contained}:
\begin{corollary}\label{cor:weak 1}
Let a random matrix $\mathcal{P}\in\mathcal{G}_{\lambda,d}$ be given. If $\mathcal P$ is a random cubature of strength $2$, then, for $d\ge 2$ and $X\in\mathscr{H}_d$,
\begin{equation}\label{eq:chrisi}
a_1\mathbb{E}\langle \mathcal{P},X\rangle \mathcal{P}  = X+a_2\trace(X) I_d,
\end{equation}
where $a_1=\frac{d(d+2)(d-1)}{-2s_1^2+2d s_2}$ and $a_2=\frac{(d+1)s_1^2-2s_2}{-2s_1^2+2d s_2}$. 
Moreover, if $\mathcal P$ is a random cubature of strength $3$, then, for $d\geq 3$ and $X_1,X_2\in\mathscr{H}_d$, 
\begin{align*}
\mathbb{E}\langle \mathcal{P},X_1\rangle  \langle \mathcal{P},X_2\rangle \mathcal{P} & =
\frac{1}{q_{3,d}}\big( \alpha_{(1,1,1)}\trace(X_1)\trace(X_2)I_d+\\
&\quad \frac{\alpha_{(2,1)}}{3}(\trace(X_1)X_2+\trace(X_2)X_1+\trace(X_1X_2)I_d )\\
&\quad \alpha_{(3)} \trace(X_1X_2)I_d\big).
\end{align*}
\end{corollary}
Note that \eqref{eq:chrisi} has been derived for $\lambda=(1,\ldots,1,0,\ldots,0)^*$ in \cite{Bachoc:2012fk}. It is also worth mentioning that $a_1$ is on the order of $d^2$ when $d$ tends to infinity and $a_2$ behaves like a constant that may depend on $k$. The coefficients $\frac{\alpha_{(2,1)}}{q_{3,d}}$, $\frac{\alpha_{(3)}}{q_{3,d}}$, and $\frac{\alpha_{(1,1,1)}}{q_{3,d}}$ behave like $1/d^{3}$ when $d$ tends to infinity.

We establish one more consequence:
\begin{corollary}\label{corol:cubature}
Suppose that $d\geq 2$. If a random matrix $\mathcal{P}\in\mathcal{G}_{\lambda,d}$ is a random cubature of strength $2$, then
\begin{equation*}
\mathcal{S}a_1\mathbb{E}(\langle \mathcal{P},X\rangle \mathcal{P})  =a_1\mathbb{E}(\langle\mathcal{P}, \mathcal{S} X\rangle \mathcal{P})  = X,\quad\text{for all } X\in\mathscr{H}_d,
\end{equation*}
where $\mathcal{S}:\mathscr{H}_d\rightarrow \mathscr{H}_d$, $X\mapsto X-\frac{a_2}{1+a_2 d}\trace(X)I_d$.
\end{corollary}
Note that $\mathcal{S}$ in Corollary \ref{corol:cubature} is a contraction, so that 
\begin{equation}\label{eq:estimate on S}
\mathcal{I}\succeq \mathcal{S}\succeq 0,\qquad \|\mathcal{S}\|_{Op}\leq 1,
\end{equation}
where $\mathcal{I}$ denotes the identity map on $\mathscr{H}_d$.

\begin{remark}
By using the theory of zonal polynomials, we have explicitly computed the $t$-th trace moments for $t=1,2,3$ to be able to verify the Theorems \ref{theorem:lower inequality} and \ref{th:dual cert}. Indeed, the trace moments are an essential ingredient in their proofs. We established the Corollaries \ref{cor:weak 1} and \ref{corol:cubature} in our more general setting. Next, we can essentially follow the lines in \cite{Gross:2013fk} with adjusted parameters and minor modifications to verify the Theorems \ref{theorem:lower inequality} and \ref{th:dual cert}, so that we have put their complete proofs in Appendix \ref{app:1} and \ref{app:2}.
\end{remark}

\section{Conclusions}\label{sec:concl}
Our results generalize findings in \cite{Gross:2013fk} from $1$-dimensional subspace measurements to the general setting of rank-$k$ positive semidefinite matrices. Moreover, we deal with cubatures in place of the required $t$-designs in \cite{Gross:2013fk}. Existence of strong $t$-designs is not fully understood yet while existence of cubatures with decent support size is rather well understood.  
In summary, our contribution is also a significant improvement for the case $k=1$ already.

Our proofs were guided by the approach in \cite{Gross:2013fk}. In our general setting, we had to compute the trace moments on $\mathcal{G}_{\lambda,d}$ for $t=1,2,3$ by applying zonal polynomials as discussed in \cite{James:1964mz}.    
%
Based on such findings, we then followed the structure in \cite{Gross:2013fk} with adjusted parameters and constants in the appendix to verify the phase retrieval results. We only explicitly addressed the real setting, but the theory of complex zonal polynomials also works in complex space with adjusted coefficients, but the asymptotics in $d$ remain the same, so that our approach covers the complex phase retrieval setting as well. 


\appendix
\section{Near isometries: proof of Theorem \ref{theorem:lower inequality}}\label{app:1}
To prove Theorem \ref{theorem:lower inequality}, we shall make use of the following deviation bound that was also used in \cite{Gross:2013fk}:
\begin{theorem}[\cite{Tropp:fk}]\label{th:tropp 2}
Let $S=\sum_{j=1}^n M_j$ be a sum of independently identically distributed $d\times d$ random matrices with zero mean and smallest eigenvalue $\lambda_{min}\geq -R$ almost surely. For $\sigma^2=\|\sum_{j=1}^n \mathbb{E}M_j^2 \|_{Op}$, the smallest eigenvalue $\Lambda_{min}$ of $S$ satisfies, for all $q\geq 0$,
\begin{equation*}
\Prob\big( \Lambda_{min} \leq -q\big) \leq d \exp( -\frac{q^2/2}{\sigma^2+Rq/3} ) \leq d\begin{cases} \exp(-3q^2/8\sigma^2),& q\leq \sigma^2/R,\\
\exp(-3q/8R),& q\geq \sigma^2/R.
\end{cases}
\end{equation*}
\end{theorem}


\begin{proof}[Proof of Theorem \ref{theorem:lower inequality}]
We make use of the mapping
\begin{equation}\label{eq:def R}
\mathcal{R}: \mathscr{H}_d\rightarrow\mathscr{H}_d, \quad X\mapsto \frac{a_1}{n}\sum_{j=1}^n \langle X,\mathcal{P}_j\rangle \mathcal{P}_j,
\end{equation}
where $a_1$ is as in Corollary \ref{cor:weak 1} and whose expectation was derived there too. Without loss of generality, we can assume $x\neq 0$. As in the proof of Proposition 9 in \cite{Gross:2013fk}, we derive
\begin{align*}
\frac{1}{a_1}\big(  1+\Lambda_{min}  \big) \|X\|^2_{F}\leq \frac{1}{n}\|\mathcal{A}_n(x)\|^2 
\end{align*}
where $\Lambda_{min}$ is the minimal eigenvalue of $\mathrm{P}_{T_x}(\mathcal{R}-\mathbb{E}\mathcal{R})\mathrm{P}_{T_x}$. Here, $\mathrm{P}_{T_x}$ is the orthogonal projector onto $T_x$, explicitly given by
\begin{equation}\label{eq:orthproj}
\mathrm{P}_{T_x}: \mathscr{H}_d\rightarrow \mathscr{H}_d,\quad X\mapsto P_{x\R}X+XP_{x\R}-\langle X,P_{x\R}\rangle P_{x\R},
\end{equation}
where $P_{x\R}=\frac{1}{\|x\|^2}xx^*$ is the orthogonal projector onto $x\R$.
Thus, we must find a lower bound on $\Lambda_{min}$.

We now split
\begin{equation*}
\mathrm{P}_{T_x}(\mathcal{R}-\mathbb{E}\mathcal{R})\mathrm{P}_{T_x} = \sum_{j=1}^n \mathcal{M}_j - \mathbb{E}\mathcal{M}_j,\quad\text{ where } \mathcal{M}_j = \frac{a_1}{n}\langle \mathrm{P}_{T_x}\cdot,\mathcal{P}_j\rangle (\mathcal{P}_j)_{T_x}.
\end{equation*}
It is fairly easy to see that $\langle X_{T_x},I_d\rangle I_{T_x}=\langle X,P_{x\R}\rangle P_{x\R}$, which implies
\begin{equation*}
-\frac{2}{n}\mathcal{I}\preceq  -\frac{1}{n}\mathcal{I} -\frac{1}{n} \langle \cdot,P_{x\R}\rangle P_{x\R} \preceq \mathrm{P}_{T_x}(\mathcal{M}_j - \mathbb{E}\mathcal{M}_j) \mathrm{P}_{T_x},
\end{equation*}
so that
\begin{equation}\label{eq:lambda min}
-\frac{2}{n}\leq \lambda_{min},
\end{equation}
where $\lambda_{min}$ is the minimal eigenvalue of $\mathcal{M}_j - \mathbb{E}\mathcal{M}_j$. We have
\begin{align*}
0 \preceq  \mathbb{E}((\mathcal{M}_j-\mathbb{E}\mathcal{M}_j)^2) \preceq \mathbb{E}(\mathcal{M}_j^2)= \frac{a_1^2}{n^2}\mathbb{E}\langle \cdot,(\mathcal{P}_j)_{T_x}\rangle\langle (\mathcal{P}_j)_{T_x},(\mathcal{P}_j)_{T_x}\rangle (\mathcal{P}_j)_{T_x}
\end{align*}
and according to $\langle (\mathcal{P}_j)_{T_x},(\mathcal{P}_j)_{T_x}\rangle=\langle \mathcal{P}_j,(\mathcal{P}_j)_{T_x}\rangle$, we can use \eqref{eq:orthproj} to derive
\begin{align*}
\langle (\mathcal{P}_j)_{T_x},(\mathcal{P}_j)_{T_x}\rangle = \trace(\mathcal{P}_j  (P_{x\R}\mathcal{P}_j+\mathcal{P}_jP_{x\R}-\langle \mathcal{P}_j,P_{x\R}\rangle P_{x\R} )\leq 2\langle \mathcal{P}_j,P_{x\R}\rangle,
\end{align*}
so that
\begin{align*}
\mathbb{E}((\mathcal{M}_j-\mathbb{E}\mathcal{M}_j)^2) \preceq \frac{2a_1^2}{n^2}\mathrm{P}_{T_x}\mathbb{E}\langle \mathrm{P}_{T_x},\mathcal{P}_j\rangle \langle \mathcal{P}_j,P_{x\R}\rangle \mathcal{P}_j.
\end{align*}
Since we have a cubature of strength $3$, we derive the estimates
\begin{align*}
\mathbb{E}((\mathcal{M}_j-\mathbb{E}\mathcal{M}_j)^2) &\preceq\frac{2a_1^2}{n^2}\mathrm{P}_{T_x}\mathbb{E}\langle \mathrm{P}_{T_x},\mathcal{P}_j\rangle \langle \mathcal{P}_j,P_{x\R}\rangle \mathcal{P}_j\\
& = \frac{2a_1^2}{n^2}\mathrm{P}_{T_x}\big[ \alpha_1 ( \mathrm{P}_{T_x}+\trace(\mathrm{P}_{T_x}\cdot)P_{x\R}+ \trace((\mathrm{P}_{T_x}\cdot)P_{x\R})I) \\
& \qquad +\alpha_2((\mathrm{P}_{T_x}\cdot) P_{x\R}+P_{x\R}\mathrm{P}_{T_x})+\alpha_3\trace(\mathrm{P}_{T_x}\cdot)I\big]\\
& =\frac{2a_1^2}{n^2}\big[ \alpha_1 ( \mathrm{P}_{T_x}+\trace((\mathrm{P}_{T_x}\cdot) P_{x\R})P_{x\R}+ \trace((\mathrm{P}_{T_x}\cdot) P_{x\R})P_{x\R}) \\
& \qquad +\alpha_2((\mathrm{P}_{T_x}\cdot) P_{x\R}+P_{x\R}\mathrm{P}_{T_x} )+\alpha_3\trace((\mathrm{P}_{T_x}\cdot) P_{x\R})P_{x\R}\big]\\
& =\frac{2a_1^2}{n^2}\big[ \alpha_1 ( \mathrm{P}_{T_x}+\trace((\mathrm{P}_{T_x}\cdot) P_{x\R})P_{x\R}+ \trace((\mathrm{P}_{T_x}\cdot) P_{x\R})P_{x\R}) \\
& \qquad +\alpha_2(\mathrm{P}_{T_x} +P_{x\R}\trace(P_{x\R}\mathrm{P}_{T_x}\cdot ))+\alpha_3\trace((\mathrm{P}_{T_x}\cdot) P_{x\R})P_{x\R}\big],
\end{align*}
where we have used \eqref{eq:orthproj} twice and $\trace(\mathrm{P}_{T_x}\cdot P_{x\R}) = \trace(\mathrm{P}_{T_x}\cdot )$. Next, we apply $\alpha_i\leq \frac{c}{d^3}$ for sufficiently large $d$ and obtain
\begin{align*}
\mathbb{E}((\mathcal{M}_j-\mathbb{E}\mathcal{M}_j)^2) &\preceq \frac{2ca_1^2}{n^2d^3}\big[ \mathrm{P}_{T_x}+\trace((\mathrm{P}_{T_x}\cdot) P_{x\R})P_{x\R}+ \trace((\mathrm{P}_{T_x}\cdot) P_{x\R})P_{x\R} \\
& \qquad +\mathrm{P}_{T_x} +P_{x\R}\trace(P_{x\R}\mathrm{P}_{T_x}\cdot )+\trace((\mathrm{P}_{T_x}\cdot) P_{x\R})P_{x\R}\big]\\
& = \frac{2ca_1^2}{n^2d^3}\big[ 4 \trace((\mathrm{P}_{T_x}\cdot) P_{x\R})P_{x\R}+2\mathrm{P}_{T_x}\big]\\
& \preceq \frac{16ca_1^2}{n^2d^3} \mathcal{I}.
\end{align*}
The rough estimate $a_1\leq c_2d^2$ implies
\begin{equation*}
\mathbb{E}((\mathcal{M}_j-\mathbb{E}\mathcal{M}_j)^2) \preceq \frac{16cc_2^2 d}{n^2}\mathcal{I}.
\end{equation*}
Let $\sigma^2:=\frac{16cc_2^2 d}{n}$, so that Theorem \ref{th:tropp 2} yields with $R=2/n$, see \eqref{eq:lambda min},
\begin{align*}
\Prob\big(  \Lambda_{min} \leq -\epsilon \big) &\leq d^2 \exp( -\frac{n\epsilon^2/2}{16cc_2^2 d+2\epsilon/3})\\
& \leq d^2 \exp( -\frac{n\epsilon^2}{32cc_2^2 d+4\epsilon/3})\leq d^2 \exp( -\frac{c_3n\epsilon^2}{d}),
\end{align*}
for all $0\leq \epsilon\leq 1\leq 32cc_2^2 d =\sigma^2/R$.

So far, we have verified that
\begin{equation*}
\frac{1}{a_1}\big(  1-\epsilon \big) \|X\|^2_{F}\leq \frac{1}{n}\|\mathcal{A}_n(X)\|^2
\end{equation*}
holds with probability of failure at most $d^2 \exp( -\frac{c_3n\epsilon^2}{d})$. If we choose $\epsilon$ fixed such that $\epsilon\leq 1-\frac{a_1}{c_2d^2}$, then we can conclude the proof. 
\end{proof}

\section{Dual certifcate: proof of Theorem \ref{th:dual cert}}\label{app:2}
We first derive a bound for $\mu^t_{\lambda,d}(xx^*)$:
\begin{proposition}\label{prop:first}
If $x\in S^{d-1}$, then we have 
\begin{equation*}
\mu^t_{\lambda,d}(xx^*)\leq \big(\frac{kt}{d}\big)^t.
\end{equation*}
\end{proposition}
\begin{proof}
This bound has been derived in \cite{Bachoc:2012fk} for $\lambda$ having $k$ ones and $d-k$ zeros. The general conditions on $\lambda$, i.e., only $k$ entries are nonzero and $\lambda_i\leq 1$, imply the statement.
\end{proof}

We shall now bound $\langle\mathcal{P},xx^*\rangle$:
\begin{proposition}\label{prop:one of the estimates}
Suppose that $x\in S^{d-1}$. If $\mathcal{P}\in\mathcal{G}_{\lambda,d}$ is a random tight $t$-fusion frame with $t\geq 1$, then we have, for all $0<r\leq 1\leq s$,
\begin{equation*}
\langle\mathcal{P},xx^*\rangle \leq (s+1)tkd^{-r}
\end{equation*}
with probability of failure at most $s^{-t}d^{-t(1-r)}$.
\end{proposition}
\begin{proof}
%
For $s\geq 1$, we estimate
\begin{align*}
\Prob\big(\langle\mathcal{P},xx^*\rangle\geq (s+1)tkd^{-r}\big) & \leq \Prob\big(\langle\mathcal{P},xx^*\rangle- \mu_\mathcal{P}(xx^*)\geq (s+1)tkd^{-r}-\frac{k}{d}\big) \\
& \leq \Prob\big(\langle\mathcal{P},xx^*\rangle- \mu_\mathcal{P}(xx^*)\geq s tkd^{-r}\big),
\intertext{where we have used Theorem \ref{th:all contained} and $\trace(D_\lambda)\leq k$. Due to Proposition \ref{prop:first}, $\tau_t:=\big(\mu_\mathcal{P}^t(xx^*)\big)^{1/t}\leq \frac{kt}{d}$ holds, so that we obtain}
\Prob\big(\langle\mathcal{P},xx^*\rangle \geq (s+1)tkd^{-r}\big) & \leq \Prob\big(\big|\langle\mathcal{P},xx^*\rangle- \mu_\mathcal{P}(xx^*) \big|\geq s d^{1-r}\tau_t\big).
\end{align*}
We can conclude the proof by applying a generalized Chebyshev inequality that was used in the proof of Lemma 13 in \cite{Gross:2013fk}, i.e.,
\begin{equation*}
\Prob\big(|\langle\mathcal{P},xx^*\rangle -\mu_\mathcal{P}(xx^*) | \geq u\tau_t\big) \leq u^{-t}
\end{equation*}
and  by choosing $u=s d^{1-r}$.
\end{proof}
To introduce a sampled truncation of the operator $\mathcal{R}$ defined in \eqref{eq:def R}, we denote the event
\begin{equation*}
E_j:=\{ \langle\mathcal{P}_j,xx^*\rangle \leq (s+1)tkd^{-r}\}
\end{equation*}
 where $\{\mathcal{P}_j\}_{j=1}^n\subset \mathcal{G}_{\lambda,d}$ are i.i.d.~according to a random tight $t$-fusion frame. The number $0 < r\leq 1$ is referred to as the \emph{truncation rate}. We also decompose a fixed $0\neq Z\in T_x$ by $Z=\lambda(xz^*+zx^*)$, where $\lambda>0$ and $z\in S^{d-1}$. For this $z$, we introduce the event
\begin{equation*}
G_j:=\{ \langle\mathcal{P}_j,zz^*\rangle \leq  (s+1)tkd^{-r}\}
\end{equation*}
and define
\begin{equation*}
\mathcal{R}:\mathscr{H}\rightarrow\mathscr{H}, \quad X\mapsto \frac{a_1}{n}\sum_{j=1}^n \langle X,\mathcal{P}_j\rangle \mathcal{P}_j,
\end{equation*}
where $a_1$ is as in Corollary \ref{cor:weak 1}, 
which is the analogue of \eqref{eq:def R} with its truncated counterpart
\begin{equation}\label{eq:def Rz}
\mathcal{R}_Z :\mathscr{H}\rightarrow\mathscr{H}, \quad X\mapsto \frac{a_1}{n}\sum_{j=1}^n 1_{E_j}1_{G_j} \langle X,\mathcal{P}_j\rangle \mathcal{P}_j.
\end{equation}
It turns out that $\mathcal{R}$ and $\mathcal{R}_Z$ are close to each other in expectation:
\begin{proposition}\label{prop:mal wieder eine prop}
For $x\in\R^d$, fix $Z\in T_x$ and let $\mathcal{R}_Z$ be as in \eqref{eq:def Rz}, where $\{\mathcal{P}_j\}_{j=1}^n\subset \mathcal{G}_{\lambda,d}$ are i.i.d.~according to a random tight $t$-fusion frame with $t\geq 2$. Then, for any sufficiently large constant $c_0$, we have
\begin{equation*}
\big\| \mathbb{E}(\mathcal{R}_Z-\mathcal{R}) \big\|_{Op} \leq c_0 s^{-t}d^{2-t(1-r)} .
\end{equation*}
\end{proposition}
\begin{proof}
We first define the auxiliar operator
\begin{equation*}
\mathcal{R}_{aux} : \mathscr{H}\rightarrow\mathscr{H},\quad X\mapsto \frac{a_1}{n}\sum_{j=1}^n 1_{E_j}\langle X,\mathcal{P}_j\rangle \mathcal{P}_j.
\end{equation*}
The triangular inequality yields
\begin{equation*}
\big\| \mathbb{E}(\mathcal{R}_Z-\mathcal{R}) \big\|_{Op}\leq \big\| \mathbb{E}(\mathcal{R}_Z-\mathcal{R}_{aux}) \big\|_{Op}+\big\| \mathbb{E}(\mathcal{R}_{aux}-\mathcal{R}) \big\|_{Op}.
\end{equation*}
Since $\|\langle X,\mathcal{P}_j\rangle \mathcal{P}_j\|\leq k\|X\|$, we obtain with Proposition \ref{prop:one of the estimates}
\begin{align*}
\big\| \mathbb{E}(\mathcal{R}_{aux}-\mathcal{R}) \big\|_{Op} &\leq \frac{a_1k}{n} \sum_{j=1}^n \Prob(E^c_j)\\
& \leq a_1 k s^{-t}d^{-t(1-r)}\\
& \leq \frac{c_0}{2} d^2 s^{-t}d^{-t(1-r)},
\end{align*}
since $a_1$ behaves like $d^2$. 
The analogue estimates for $\big\| \mathbb{E}(\mathcal{R}_{aux}-\mathcal{R}) \big\|_{Op} $ using $c_0/2$ conclude the proof.
\end{proof}

Let $ \mathrm{P}_{T_x}:\mathscr{H}_d\rightarrow T_x$ be the orthogonal projector onto $T_x$, i.e., $\mathrm{P}_{T_x}(Y)=Y_{T_x}$, and $\mathcal{P}_{T^\perp_x}$ the orthogonal projector onto the orthogonal complement of $T^\perp_x$.
\begin{proposition}\label{prop:golf fundamentals}
For $x\in\R^d$, fix $Z\in T_x$ and let $\mathcal{R}_Z$ be as in \eqref{eq:def Rz}, where $\{V_j\}_{j=1}^n\subset \mathcal{G}_{k,d}$ are i.i.d.~according to a cubature of strength $t\geq 3$, and the truncation rate is supposed to satisfy $r\leq 1-2/t$. Then there is a constant $c_1>0$ such that, for $1/c_0\leq A\leq 1$ and $\sqrt{2}A\leq B$,
\begin{align}
\label{al:1}\big\|  \mathcal{P}_{T^\perp_x} \mathcal{S} \mathcal{R}_Z Z   \big\|_{Op} & \leq A \|Z\|_F, \\
\label{al:2}\big\|  \mathrm{P}_{T_x}( \mathcal{S} \mathcal{R}_Z -\mathcal{I}) Z   \big\|_{F} & \leq B \|Z\|_F,
\end{align}
hold with probability of failure at most $d\exp(-c_1\frac{n A}{t d^{2-r}})$.
\end{proposition}
For the proof of the above proposition, we need the following concentration bound from \cite{Tropp:2011fk,Gross:2011fk}
\begin{theorem}[\cite{Tropp:2011fk,Gross:2011fk}]\label{th:2 that we need}
Consider a finite sequence $\{M_j\}_{j=1}^n$ of independent, random self-adjoint operators on $\C^d$. Assume that $\mathbb{E}M_j=0$ and $\|M_j\|_{Op}\leq R$ almost surely and let $\sigma^2=\| \sum_{j=1}^n\mathbb{E}M_j^2 \|_{Op}$. Then we have, for all $q\geq 0$,
\begin{equation*}
\Prob\big( \|\sum_{j=1}^n M_j\|_{Op}\geq q \big) \leq d \exp( -\frac{q^2/2}{\sigma^2+Rq/3})\leq \begin{cases}
d\exp(-3q^2/8\sigma^2),& q\leq \sigma^2/R,\\
d\exp(-3q/8R),& q\geq  \sigma^2/R.\\
\end{cases}
\end{equation*}

\end{theorem}

\begin{proof}[Proof of Proposition \ref{prop:golf fundamentals}]
Without loss of generality, we can assume that $Z=q(zx^*+xz^*)$ with $z\in S^{d-1}$ and $0<q\leq 1$.

As in \cite{Gross:2013fk}, we have with \eqref{eq:estimate on S} $\|\mathcal{S}\|_{Op}\leq 1$, and we can estimate with Proposition \ref{prop:mal wieder eine prop}
\begin{align*}
\big\|  \mathcal{P}_{T^\perp_x} \mathcal{S} \mathcal{R}_Z Z   \big\|_{Op}  
& \leq \big\| (\mathcal{R}_Z-\mathbb{E}\mathcal{R}_Z)Z   \big\|_{Op} +c_0s^{-t}d^{2-t(1-r)}\\
& \leq \big\| (\mathcal{R}_Z-\mathbb{E}\mathcal{R}_Z)Z   \big\|_{Op}+ c_0s^{-3}\\
& \leq \big\| (\mathcal{R}_Z-\mathbb{E}\mathcal{R}_Z)Z   \big\|_{Op}+1/c_0^2\\
&\leq \big\| (\mathcal{R}_Z-\mathbb{E}\mathcal{R}_Z)Z   \big\|_{Op}+A/c_0,
\end{align*}
where we have chosen $s=c_0$. As in \cite{Gross:2013fk}, we obtain in a similar fashion
\begin{align*}
\big\|  \mathrm{P}_{T_x} \mathcal{S} (\mathcal{R}_Z-\mathcal{I}) Z   \big\|_{F}  & \leq \sqrt{2}\big\| (\mathcal{R}_Z-\mathbb{E}\mathcal{R}_Z)Z   \big\|_{Op}+A/c_0,
\end{align*}

We define the event
\begin{equation*}
E := \{\big\| (\mathcal{R}_Z-\mathbb{E}\mathcal{R}_Z)Z   \big\|_{Op}\leq A-A/c_0  \},
\end{equation*}
so that $A$ and $B$ are chosen such that it boils down to bound the probability of $E^c$. As in \cite{Gross:2013fk}, we define
\begin{equation*}
(\mathcal{R}_Z-\mathbb{E}\mathcal{R}_Z)Z = \sum_{j=1}^n(M_j-\mathbb{E}M_j),\quad\text{where} \quad M_j = \frac{a_1}{n}1_{E_j}1_{G_j} \langle Z,\mathcal{P}_j\rangle \mathcal{P}_j
\end{equation*}
and estimate analoguously
\begin{align*}
\|M_j\|_{Op} & \leq \frac{a_1}{n}1_{E_j}1_{G_j} |\langle Z,\mathcal{P}_j\rangle|\\
& \leq  \frac{a_1}{n}1_{E_j}1_{G_j} 2|x^*\mathcal{P}_j z| \\
&\leq  \frac{a_1}{n}1_{E_j}1_{G_j} 2\sqrt{\langle \mathcal{P}_j,xx^*\rangle \langle \mathcal{P}_j,zz^*\rangle|} \\
&\leq \frac{a_1}{n}2(s+1)tkd^{-r},
\end{align*}
where we have used the definitions of $E_j$ and $G_j$. We fix $s$ and knowing that $a_1$ grows like $d^2$, we can further derive
\begin{equation*}
\|M_j\|_{Op}\leq \frac{c_3}{n}td^{2-r}=:\tilde{R}.
\end{equation*}
Next, we estimate with Corollary \ref{cor:weak 1}, $\trace(Z)\leq \sqrt{2}\|Z\|_{F}$, $Z\preceq I_d$, $Z^2\preceq \|Z\|_{F}I_d$, and $\|Z\|_{F}\leq 1$,
\begin{align*}
\mathbb{E}(M_j-\mathbb{E}M_j)^2 &\preceq \mathbb{E}M_j^2\\
& \preceq \frac{a_1^2}{n^2}\mathbb{E}\langle Z,\mathcal{P}_j\rangle^2 \mathcal{P}_j\\
&\preceq\frac{a_1^2}{n^2}\big[\alpha_1 ( 2\trace(Z)Z+ \trace(Z^2)I_d) +2\alpha_2Z^2+\alpha_3\trace(Z)^2I_d\big]\\
&\preceq\frac{a_1^2}{n^2}\big[\alpha_1 ( 2\sqrt{2}\|Z\|_{F}Z+ \trace(Z^2)I_d) +2\alpha_2\|Z\|_{F}I_d+\alpha_3 2 \|Z\|_{F}^2I_d\big]\\
&\preceq\frac{a_1^2}{n^2}\big[\alpha_1 ( 2\sqrt{2}I_d+ I_d) +2\alpha_2I_d+\alpha_3 2 I_d\big]\preceq\frac{cd}{n^2} I_d,
\end{align*}
where $c>0$ is some constant independent of $d$ and we used that $a_1$ and $\alpha_i$ can be estimated by a constant times $d^2$ and $1/d^3$, respectively. We can deduce that
\begin{equation*}
\big\|\sum_{j=1}^n \mathbb{E}(M_j-\mathbb{E}M_j^2)^2\big\|_{Op} \leq n \max_{j=1,\ldots,n} \|\mathbb{E}M_j^2 \|_{Op} \leq \frac{cd}{n}=:\sigma^2.
\end{equation*}
Now, we can choose a sufficiently large constant $c_2\geq 1$ such that the definition $R:=c_2\tilde{R}$ yields
\begin{equation*}
\frac{\sigma^2}{R}\leq \frac{cd^{r-1}}{c_2c_3t}\leq \tilde{q}A,
\end{equation*}
with some $\tilde{q}<1$. 
As in \cite{Gross:2013fk}, an application of Theorem \ref{th:2 that we need} with $q=\tilde{q}A$ concludes the proof.
\end{proof}

We have now completed the preparations for the proof of Theorem \ref{th:dual cert}:
\begin{proof}[Proof of Theorem \ref{th:dual cert}]
We construct the dual certificate in a recursive manner and begin with $Y_0=0$. Suppose that $Y_i$ is constructed, then we put $Q_i := xx^* - (Y_i)_{T_x} \in T_x$. We choose $n_i$ subspaces independently and identically distributed according to the cubature $V$. Let $\mathcal{R}_{Q_{i-1}}$ be the operator defined in \eqref{eq:def Rz}. We define
\begin{equation*}
A:= 1/c_0,\qquad B := \sqrt{2}A
\end{equation*}
and check whether \eqref{al:1} and \eqref{al:2} are satisfied. If so, let $\mathcal{R}^{(i)}_{Q_{i-1}}:=\mathcal{R}_{Q_{i-1}}$,
\begin{equation*}
Y_i:= \mathcal{S} \mathcal{R}^{(i)}_{Q_{i-1}} (xx^* - Y_{i-1})_{T_x} + Y_{i-1},
\end{equation*}
and we proceed to step $i+1$. If one of the bounds \eqref{al:1} and \eqref{al:2} does not hold, then we repeat the $i$-th step with a new batch of $n_i$ subspaces. We denote the probability of having to repeat the $i$-th step by $p_i$ and the eventual number of repetitions by $r_i$. For $l:=\lceil \log_{1/B}(d) \rceil +2$, we define $Y:=Y_l$. Analogously to \cite{Gross:2013fk}, we derive
\begin{align*}
 \|Y_{T_x} - xx^*\|_{F} \leq \frac{2}{d}A^2 &= \frac{2}{c_0^2 d},\\
 \|Y_{T_x^\perp} \|_{Op} &\leq  \frac{A}{1-\sqrt{2}A}=\frac{c_0}{c_0(c_0-\sqrt{2})}\leq \frac{1}{c_0-\sqrt{2}}.
\end{align*}
In order to estimate the probability that the total number of measurements $\sum_{i=1}^l n_i r_i$ exceeds the bound in \eqref{eq:final the}, we first apply Proposition \ref{prop:golf fundamentals} to obtain
\begin{equation*}
p_i \leq d\exp(-c_1\frac{n_i A}{t d^{2-r}})\leq d\exp(-c_1\frac{n_i}{c_0t d^{2-r}}).
\end{equation*}
To get the exact point of contact with the proof in \cite{Gross:2013fk}, we choose
\begin{equation*}
n_i=3\frac{c_0}{c_1}t d^{2-r}\log(d),
\end{equation*}
which yields
\begin{equation*}
p_i \leq e^{-3}\leq 1/20.
\end{equation*}
This is the same bound as in \cite{Gross:2013fk}. 
The remaining part of the proof is based on a concentration bound for binomial random variables and directly follows the lines in \cite{Gross:2013fk}, so that we omit the details.


\end{proof}

\section*{Acknowledgements}
ME and MG are funded by the Vienna Science and Technology Fund (WWTF) through project VRG12-009. FK is partially supported by  Mathematisches Forschungsinstitut Oberwolfach (MFO). This research was partially carried out at MFO, supported by FK's Oberwolfach Leibniz Fellowship.

\bibliographystyle{amsplain}
\bibliography{../biblio_ehler2}

\end{document}